\documentclass[12pt,a4paper,reqno]{amsart}
\usepackage[english]{babel}
\usepackage[applemac]{inputenc}
\usepackage[T1]{fontenc}
\usepackage{palatino}
\usepackage{amsmath}
\usepackage{amssymb}
\usepackage{amsthm}
\usepackage{amsfonts}
\usepackage{graphicx}
\usepackage{mathtools}

\usepackage[colorlinks = true, citecolor = black]{hyperref}
\author{Henri Martikainen and Tuomas Orponen}
\title[Jones' square function and $1$-rectifiability]{Boundedness of the density normalised Jones' square function does not imply $1$-rectifiability}
\address{University of Helsinki, Department of Mathematics and Statistics}
\thanks{H.M. is supported by the Academy of Finland through the grant
Multiparameter dyadic harmonic analysis and probabilistic methods. T.O. is supported by the Academy of Finland through the grant Restricted families of projections and connections to
Kakeya type problems. Both authors are members of the Finnish Centre of Excellence in Analysis and Dynamics Research.}
\email{henri.martikainen@helsinki.fi,  tuomas.orponen@helsinki.fi}

\newcommand{\R}{\mathbb{R}}
\newcommand{\N}{\mathbb{N}}

\newcommand{\calD}{\mathcal{D}}
\newcommand{\calH}{\mathcal{H}}

\newcommand{\calQ}{\mathcal{Q}}

\newcommand{\spt}{\operatorname{spt}}

\newcommand{\diam}{\operatorname{diam}}
\newcommand{\card}{\operatorname{card}}

\numberwithin{equation}{section}

\theoremstyle{plain}
\newtheorem{thm}[equation]{Theorem}

\newtheorem{lemma}[equation]{Lemma}

\theoremstyle{definition}

\newtheorem{definition}[equation]{Definition}

\theoremstyle{remark}
\newtheorem{remark}[equation]{Remark}

\begin{document}

\begin{abstract} 
Recently, M. Badger and R. Schul proved that for a $1$-rectifiable Radon measure $\mu$, the density weighted Jones' square function
$$
J_{1}(x) = \mathop{\sum_{Q \in \calD}}_{\ell(Q) \leq 1} \beta_{2,\mu}^{2}(3Q)\frac{\ell(Q)}{\mu(Q)} 1_{Q}(x)
$$
is finite for $\mu$-a.e. $x$. Answering a question of Badger--Schul, we show that the converse is not true.
Given $\epsilon > 0$, we construct a Radon probability measure on $[0,1]^{2} \subset \mathbb{R}^{2}$ with the properties that $J_{1}(x) \leq \epsilon$ for all $x \in \operatorname{spt} \mu$, but nevertheless
the $1$-dimensional lower density of $\mu$ vanishes almost everywhere. In particular, $\mu$ is purely $1$-unrectifiable.
\end{abstract}

\maketitle

\section{Introduction}
A Radon measure $\mu$ in $\R^2$ is $1$-rectifiable if there exist countably many Lipschitz maps $f_i\colon \R \to \R^2$ such that
$$
\mu\Bigg( \R^2 \setminus \bigcup_i f_i(\R)\Bigg) = 0.
$$
Recent years have seen lively interest in attempting to characterise the rectifiability of general Radon measures in terms of \emph{$\beta$-numbers}, originally defined by P. Jones, G. David and S. Semmes. The existence of such a characterisation was conjectured by P. Jones around 2000. We start by mentioning a three-paper series of M. Badger and R. Schul \cite{BS1,BS2, BS3}, where the authors study the connection between $1$-rectifiability and the boundedness of certain square functions, usually nicknamed \emph{Jones' square functions}. A natural example of these objects is the following function $J_{1} := J_{1,\mu}$, the \emph{density normalised Jones' square function}:
$$
J_1(x) := \mathop{ \sum_{Q \in \calD}}_{\ell(Q) \le 1} \beta_{2,\mu}^{2}(3Q)\frac{\ell(Q)}{\mu(Q)} 1_{Q}(x), \quad x \in \R^{2},
$$
where $\calD$ is the standard dyadic grid in $\R^2$, $\mu$ is a Radon measure, and the $\beta$-numbers are defined as follows.
\begin{definition}[$\beta$-numbers]\label{betas}
Let $\mu$ be a Radon measure on $\R^{2}$. For a square $Q \subset \R^{2}$ with $\mu(Q) > 0$, we define the number $\beta_{2,\mu}(Q)$ by
$$
\beta_{2,\mu}(Q) = \inf_{L} \left[ \frac{1}{\mu(Q)} \int_{Q} \left(\frac{d(y,L)}{\diam(Q)} \right)^{2} \, d\mu(y) \right]^{1/2},
$$
where the $\inf$ is taken over all (affine) lines $L \subset \R^{2}$.
\end{definition}
\noindent This $L^{2}$ definition of $\beta$-numbers is due to G. David and S. Semmes \cite{DS1,DS3}, and the
definition of the density normalised Jones' square function $J_1$ appears in the papers by Badger--Schul \cite{BS1,BS2, BS3}.

The validation for $J_{1}$ is, no doubt, the following theorem of Badger--Schul \cite{BS1}: if $\mu$ is $1$-rectifiable, then
$$
J_1(x) < \infty
$$
for $\mu$-a.e. $x$. This indicates that the pointwise $\mu$-a.e. boundedness of $J_{1}$ could, potentially, characterise the $1$-rectifiability of a general Radon mesure $\mu$, in the spirit of Jones' conjecture; in fact, Badger and Schul manage to prove this in \cite{BS3} under the assumption that $\mu$ is pointwise doubling. In the present paper, we disprove the conjecture for general measures: the pointwise boundedness of $J_{1}$ does not imply $1$-rectifiability. In fact, the boundedness of $J_{1}$ does not even imply that $\mu$ has non-vanishing $1$-dimensional lower density
$$
\Theta_{\ast}^{1}(\mu,x) := \liminf_{r \to 0} \frac{\mu(B(x,r))}{2r}
$$
in a set of positive measure. Here is the precise statement:
\begin{thm}\label{main2} Given $\epsilon > 0$, there exists a Radon probability measure $\mu$ supported on $[0,1]^{2} \subset \R^{2}$ with the following properties:
\begin{enumerate}
\item $J_{1}(x) \leq \epsilon$ for all $x \in \spt \mu$.
\item $\Theta_{\ast}^{1}(\mu,x) = 0$ for $\mu$-a.e. $x$.
\end{enumerate}
In particular, $\mu$ is purely $1$-unrectifiable.
\end{thm}
The fact that (ii) implies pure $1$-unrectifiability follows from Lemma 2.7 in Badger and Schul's paper \cite{BS1}, which states that $1$-rectifiable measures have positive lower $1$-density almost everywhere.   

We mention a few further developments. For general measures, Badger and Schul \cite{BS3} were able to get a positive result by considering a somewhat larger square function $\tilde{J}_{1}$,
where the only difference to $J_{1}$ is that the $\beta$-numbers are replaced by certain larger versions. If one makes the \emph{a priori} assumption $\mu \ll \calH^{1}$, then a full characterisation of rectifiability is available, thanks to Tolsa \cite{To1} and Azzam-Tolsa \cite{AT1}: in their theorem, Azzam and Tolsa consider a variant of $J_{1}$, which works precisely in the case $\mu \ll \calH^{1}$ but has no chance to characterise the $1$-rectifiability of general Radon measures.

The proof Theorem \ref{main2} is rather technical, so for the reader's convenience we first prove the following simpler version:
\begin{thm}\label{main} Given $\epsilon > 0$, there exists a Radon probability measure $\mu$ supported on $[0,1]^{2} \subset \R^{2}$ with the following properties:
\begin{itemize}
\item[(i)] The $\beta$-numbers associated with $\mu$ satisfy
\begin{displaymath}  \sum_{Q \in \calD} \beta_{2,\mu}^{2}(3Q)\frac{\ell(Q)}{\mu(Q)} 1_{Q}(x) \leq \epsilon, \quad x \in \spt \mu.  \end{displaymath}
\item[(ii)] If $\Gamma$ is any $\calH^{1}$-measurable set with $\calH^{1}(\Gamma) \leq 1$, we have
\begin{displaymath} \mu(\Gamma) \leq \epsilon. \end{displaymath}
\end{itemize}
\end{thm}
This already shows that the condition (i) is not sufficient for a traveling salesman type theorem for general measures. 
Such a result was obtained by Badger--Schul \cite{BS3} by replacing the $\beta$-numbers in (i) by some enlarged $\beta$-numbers called $\tilde \beta$.
Their result says that if
\begin{equation*} \tilde \beta(\mu)^{2} := \sum_{Q \in \calD} \tilde \beta_{2,\mu} (3Q)^{2} \ell(Q) < \infty, \end{equation*}
then $\mu$ almost all of $\R^{2}$ can be covered by a single curve $\Gamma$, whose length respects the upper bound $\calH^{1}(\Gamma) \lesssim \diam \spt \mu + \tilde \beta_{2}(\mu)^{2}$.
Theorem \ref{main} implies that this fails if $\tilde \beta(\mu)^2$ is replaced with
$$
 \beta(\mu)^{2} := \sum_{Q \in \calD} \beta_{2,\mu} (3Q)^{2} \ell(Q).
$$
For the original Jones' traveling salesman for compact sets, see \cite{Jo1}. Lerman \cite{Le1} has considered questions related to this paper for non-density normalised variants of $J_{1}$. In this setting, he has examples similar to Theorem \ref{main}, see Section 5 of \cite{Le1}. 

As the reader will note, the proofs of Theorems \ref{main} and \ref{main2} make good use of the special properties of dyadic squares. It is fair to ask, whether such dyadic magic is actually necessary for the examples to work. Our third result should alleviate these concerns: it is a version of Theorem \ref{main}, where the dyadic square $J_{1}$ is replaced by a continuous analogue. It is very likely that Theorem \ref{main2} could also be adapted to this setting, but the technical details are, no doubt, easier to handle in the dyadic world.
\begin{thm}\label{main3} Given $\epsilon > 0$, there exists a Radon probability measure $\mu$ supported on $[0,1]^{2} \subset \R^{2}$ with the following properties:
\begin{itemize}
\item[(i)] The $\beta$-numbers associated with $\mu$ satisfy
\begin{displaymath} \int_{0}^{\infty} \beta^{2}_{2,\mu}(B(x,r)) \frac{dr}{\mu(B(x,r))} \leq \epsilon \quad \text{for all } x \in \spt \mu.  \end{displaymath}
\item[(ii)] If $\Gamma$ is any $\calH^{1}$-measurable set with $\calH^{1}(\Gamma) \leq 1$, we have
\begin{displaymath} \mu(\Gamma) \leq \epsilon. \end{displaymath}
\end{itemize}
\end{thm}
Above, the $\beta$-numbers associated to balls are defined simply by replacing all occurrences of $Q$ by $B(x,r)$ in Definition \ref{betas}.

\subsection{Notation} In this paper, $B(x,r)$ stands for an open ball of radius $r > 0$ centred at $x \in \R^{2}$. If $Q$ is a dyadic square, the notation $3Q$ stands for the square, which has the same centre as $Q$ but is dilated by a factor of three; in other words, $3Q$ is the union of $Q$ with its eight dyadic neighbours. For non-negative real numbers $A,B$, the notation $A \lesssim B$ means that $A \leq CB$ for some absolute constant $C \geq 1$. The two-sided inequality $A \lesssim B \lesssim A$ is abbreviated to $A \sim B$. 

\section{Proof of Theorem \ref{main}}
In this section we consider the proof of the simpler version of our main Theorem \ref{main2}, namely Theorem \ref{main}.
This shows some of the essential ideas, but avoids the technicalities imposed by the iterative construction needed in Theorem \ref{main2}.

We start with the definition of an auxiliary measure inside a fixed dyadic square.
\begin{definition}\label{auxMeasure} Let $1 \geq r_{0} > r_{1} > \ldots > r_{N} > 0$ be a finite sequence of dyadic radii, let $m > 0$ be a "mass", and let $\delta > 0$ be a "density". Assume that $N \leq m/(\delta r_{0})$. We define a measure $\nu = \nu((r_{j}),m,\delta)$ supported on $[0,r_{0}]^{2}$ as follows:
\begin{displaymath} \nu := \sum_{j = 1}^{N} \delta \cdot \calH^{1}|_{E_{j}} + \Theta \cdot \calH^{1}|_{H}. \end{displaymath}
Here $E_{j} = [0,r_{0}] \times \{r_{j}\}$ for $1 \leq j \leq N$, and $H = [0,r_{0}] \times \{0\}$. The number $\Theta \geq 0$ solves the equation
\begin{displaymath} N\delta r_{0} + \Theta r_{0} = m. \end{displaymath}
In particular, $\nu(\R^{2}) = m$.  \end{definition}

We set
$$
\mu := \nu((r_{j})_{j=0}^N,1,\epsilon),
$$
where $(r_{j})_{j = 0}^{N}$ is a suitable sequence of dyadic radii to be specified later, with
$$
r_{0} = c\epsilon^{3} \quad \text{and} \quad N = \frac{1-\epsilon}{\epsilon r_{0}}
$$
for some small constant $c > 0$. For simplicity, we assume $N$ to be an integer (otherwise take the ceiling function).

We first verify (ii) (with $2\epsilon$ instead of $\epsilon$). Suppose that $\Gamma$ satisfies $\calH^{1}(\Gamma) \leq 1$. Then,
\begin{align*}
\mu(\Gamma) \le \epsilon\calH^{1}(\Gamma) + \Theta \calH^{1}(H) \le 2\epsilon.
\end{align*}

It remains to show (i), that is to show that
\begin{displaymath} \sum_{Q \in \calD} \beta_{2,\mu}^{2}(3Q) \frac{\ell(Q)}{\mu(Q)} 1_{Q}(x) \le \epsilon, \qquad x \in \spt \mu. \end{displaymath}
We begin by treating the case where $x \in E_{j}$ for some $1 \leq j \leq N$. Let $Q \in \calD$ be a dyadic square containing $x$. The three essential cases of the proof are illustrated in Figure \ref{fig3}.
\begin{figure}[h!]
\begin{center}
\includegraphics[scale = 0.8]{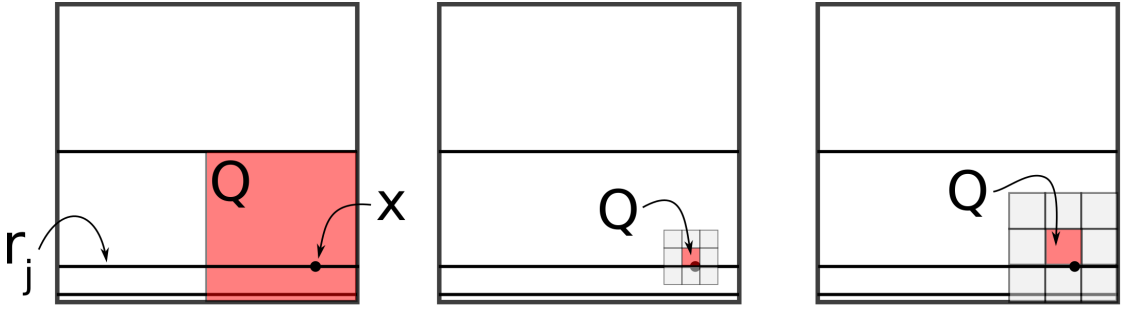}
\caption{The three cases: $\ell(Q) > r_{j}$, $\ell(Q) < r_{j}$, or $\ell(Q) = r_{j}$.}\label{fig3}
\end{center}
\end{figure}

Suppose first that $r_j < \ell(Q) \leq c_{0}r_{0}$ for some absolute constant $c_{0} > 0$ ($c_{0} = 1/10$ is more than good enough). It follows that the lower edge $J_Q$ of $Q$
has to be contained in $H = [0,r_{0}] \times \{0\}$, since the second coordinate of $x$ is $r_j < \ell(Q)$, and so $J_Q \subset \R \times \{0\}$. Since moreover $\ell(Q) \le r_{0}$ and
$r_{0}$ is a dyadic number, we have that $Q \subset [0, r_{0})^2 $, and so $J_Q \subset H$.
Using this,
\begin{equation}\label{form28} \mu(3Q) \geq \mu(Q) \geq \mu(Q \cap H) = \Theta \cdot \calH^{1}(J_Q) = \Theta \cdot \ell(Q) = \frac{\epsilon \ell(Q)}{r_{0}}. \end{equation}
We, temporarily, further assume that $r_{j} < \ell(Q) \leq c_0r_{j - 1}$; this implies that $3Q$ does not meet the segment $E_{j - 1}$ (for $c_{0} > 0$ small enough), and hence by testing the $\beta$-numbers
with the line $L = \R \times \{0\}$ we get the following estimate:
\begin{align*}
\sum_{r_{j} < \ell(Q) \leq c_0r_{j - 1}} \beta^{2}_{2,\mu}(3Q) \frac{\ell(Q)}{\mu(Q)} 1_{Q}(x) & \lesssim \sum_{r_{j} < \ell(Q) \leq c_0r_{j - 1}}  \frac{1_{Q}(x)r_{0}^{2}}{\epsilon \ell(Q)} \sum_{i = 1}^{N} \int_{3Q \cap E_{i}} \left(\frac{r_{i}}{\ell(Q)}\right)^{2} \, d\calH^{1}\\
&\lesssim \sum_{r_{j} < \ell(Q) \leq c_0r_{j - 1}} 1_{Q}(x) \frac{r_{0}^{2}}{\epsilon \ell(Q)^{2}} \sum_{i = j}^{N} r_{i}^{2} \sim \frac{r_{0}^{2}}{\epsilon}.
\end{align*}
The same estimate also holds, if one sums over cubes $Q$ such that $x \in Q$ and $c_0r_{i} \leq \ell(Q) \leq c_0 r_{i - 1}$ for any $1 \leq i \leq j - 1$ (note also that $\ell(Q) > r_j$ in this case
if the sequence $(r_j)$ decays rapidly enough, which guarantees that we have the lower bound for $\mu(Q)$ exactly as above).
 There are at most $N \le 1/(\epsilon r_{0})$ such indices $i$, so it follows that
\begin{displaymath} \sum_{r_{j} < \ell(Q) \leq c_0r_{0}} \beta^{2}_{2,\mu}(3Q) \frac{\ell(Q)}{\mu(Q)} 1_{Q}(x) \lesssim \frac{r_{0}}{\epsilon^{2}} = c\epsilon. \end{displaymath}
This completes the case $r_{j} < \ell(Q) \leq c_{0}r_{0}$.

If $\ell(Q) \le r_{j}$, then the lower edge of $Q$ is contained in $E_j$. In the case of strict inequality $\ell(Q) < r_j$, we have $\ell(Q) \le r_j/2$, and so
$3Q$ intersects \textbf{no} line segments besides $E_{j}$ (if the sequence $(r_{j})$ decays rapidly enough). It follows that $\beta^{2}_{2,\mu}(3Q) = 0$ in case $\ell(Q) < r_{j}$, and there is nothing to prove.

Finally, if $\ell(Q) = r_{j}$, then $3Q$ just intersects $H$, and in fact $\calH^{1}(3Q \cap H) \sim \ell(Q)$.
It follows that
\begin{displaymath} \mu(3Q) \gtrsim \frac{\epsilon \ell(Q)}{r_{0}}. \end{displaymath}
For $\mu(Q)$ we get a worse estimate, because $Q$ does not meet $H$:
\begin{displaymath} \mu(Q) \geq \epsilon \calH^{1}(E_{j} \cap Q) = \epsilon\ell(Q). \end{displaymath}
Assuming that the decay of the numbers $(r_{i})$ is rapid enough, the square $3Q$ does not meet the line segments $E_{i}$ for $i \leq j - 1$, so we get the following estimate for the number $\beta^{2}_{2,\mu}(3Q)$ (using $L = \R \times \{0\}$ as the testing line):
\begin{displaymath} \beta^{2}_{2,\mu}(3Q) \lesssim \frac{r_{0}}{\ell(Q)} \sum_{i = 1}^{N} \int_{3Q \cap E_{i}} \left(\frac{r_{i}}{\ell(Q)}\right)^{2} \, d\calH^{1} \lesssim \frac{r_{0}}{\ell(Q)^{2}} \sum_{i = j}^{N} r_{i}^{2} \lesssim r_{0}. \end{displaymath} 
Consequently,
\begin{displaymath} \sum_{\ell(Q) = r_{j}} \beta^{2}_{2,\mu}(3Q) \frac{\ell(Q)}{\mu(Q)} 1_{Q}(x) \lesssim \frac{r_{0}}{\epsilon} = c\epsilon^2. \end{displaymath}
This completes the proof in the case $x \in E_{j}$, except for the squares $Q$ with $\ell(Q) \geq c_{0}r_{0}$. In that case, the lower edge of $Q$ is again contained on $\R \times \{0\}$, and $\mu(Q) \sim 1 \sim \mu(3Q)$. Thus,
\begin{displaymath} \beta^{2}_{2,\mu}(3Q) \lesssim \epsilon \sum_{i = 1}^{N} \int_{E_{i}} \left(\frac{r_{i}}{\ell(Q)} \right)^{2}\,d\calH^{1} \lesssim \frac{\epsilon r_{0}^{3}}{\ell(Q)^{2}}, \end{displaymath}
and so
\begin{displaymath} \sum_{\ell(Q) \geq c_{0}r_{0}} \beta_{2,\mu}^{2}(3Q) \frac{\ell(Q)}{\mu(Q)} 1_{Q}(x) \lesssim \epsilon r_{0}^{3} \sum_{\ell(Q) \geq c_{0}r_{0}} \frac{1_Q(x)}{\ell(Q)} \sim \epsilon r_{0}^{2} = c^{2}\epsilon^{7}. \end{displaymath}

It remains to consider the case $x \in H$. If $Q \ni x$ is a dyadic square of side-length $\ell(Q) \leq c_{0}r_{0}$, then the lower edge of $Q$ is obviously contained in $H$, and we have the estimates \eqref{form28} for $\mu(Q)$ and $\mu(3Q)$. This allows us to make the same computations as in the case $\ell(Q) > r_{j}$ above, and the $\beta$-sum is again bounded by $\lesssim r_{0}/\epsilon^{2} = c\epsilon$. The case $\ell(Q) \geq c_{0}r_{0}$ is exactly the same as above. The proof of Theorem \ref{main} is complete.

\begin{remark}
For our examples, and for the theory in whole, it is irrelevant whether the dyadic squares in $\calD$ are taken to be half-open or closed.
Only minor modifications are required in the arguments.
\end{remark}

\section{Proof of Theorem \ref{main2}}

We define the measure $\mu$ as a weak limit of certain measures $\mu_{n}$. Assume that $\epsilon \in (0,2^{-10})$ is a dyadic number, and set
\begin{displaymath} \mu_{1} := \nu((r_{1,j}),1,\epsilon), \end{displaymath} 
where $(r_{1,j})_{j = 0}^{N_{1}}$ is a suitable sequence of dyadic radii to be specified later, with
\begin{displaymath} r_{1,0} = c\epsilon^{10} \quad \text{and} \quad N_{1} = \frac{1}{2\epsilon r_{1,0}} \end{displaymath}
for some small constant $c > 0$. Then $N_{1}\epsilon r_{1,0} = 1/2 \leq 1$, so the upper bound condition for $N$ from Definition \ref{auxMeasure} is satisfied. Next, assume that the probability measure $\mu_{n}$ has already been defined for some $n \geq 1$, and $\mu_{n}$ has the form
\begin{displaymath} \mu_{n} = \sum_{i \in I_{n}} \omega_{i}^{n} \cdot \calH^{1}|_{J_{i}^{n}}, \end{displaymath}
where the numbers $\omega_{i}^{n}$ are positive dyadic numbers, and the sets $J_{i}^{n}$ are line segments of the form $J_{i}^{n} = [a^{n}_{i},b^{n}_{i}] \times \{c^{n}_{i}\}$, where $a^{n}_{i},b^{n}_{i},c^{n}_{i}$ are dyadic rationals. Further, assume that $b^{n}_{i} - a^{n}_{i} = d_{n}$ is a positive dyadic rational depending only on $n$. 

To define $\mu_{n + 1}$, choose an equally spaced set of dyadic rationals from the interval $[0,d_{n}]$, with "spacing" $\sigma_{n} > 0$, and call this set $D_{n}$. More precisely,
\begin{displaymath} D_{n} = \{0,\sigma_{n},\ldots,d_{n} - \sigma_{n}\}, \end{displaymath}
so that $|D_{n}| = d_{n}/\sigma_{n}$. 

The size of the number $\sigma_{n}$ is determined by the mutual distances between distinct line segments $J_{i}^{n}$. In fact, if the minimal distance between any distinct pair of line segments $J_{i_{1}}^{n},J_{i_{2}}^{n}$ is $\Delta_{n}$, then we set
\begin{equation}\label{form19} \sigma_{n} := \frac{\Delta_{n}}{1000}. \end{equation}
Next, for a fixed segment $J_{i}^{n}$ with $i \in I_{n}$, write 
\begin{equation}\label{form18} m_{n}^{i} := \frac{\omega_{i}^{n}d_{n}}{|D_{n}|} = \sigma_{n}\omega_{i}^{n}. \end{equation} 
Then, pick a decreasing sequence of dyadic radii $(r_{n + 1,j})_{j = 0}^{N^{i}_{n + 1}}$, where
\begin{equation}\label{form20} N^{i}_{n + 1} = \frac{2^{n - 1}m^{n}_{i}}{\epsilon r_{n + 1,0}}, \end{equation}
and set
\begin{displaymath} \nu^{i}_{n + 1} := \nu((r_{n + 1,j})_{j = 0}^{N^{i}_{n + 1}},m_{n}^{i},\epsilon/2^{n}). \end{displaymath}
Note that the choice of $N_{n + 1}^{i}$ is a dyadic positive integer, if $r_{n + 1,0}$ is a small enough dyadic number, where "small enough" only depends on the generation $n$ quantities $m_{n}^{i}$. Note that only the length of the sequence $(r_{n + 1,j})^{N^{i}_{n + 1}}_{j = 0}$ depends on the index $i \in I_{n}$, but the numbers themselves are all part of some single rapidly decreasing infinite sequence, which only depends on $n + 1$. We assume that
\begin{equation}\label{form9} r_{n + 1,0} \leq \min \left\{ \frac{(\epsilon \Delta_{n})^{10}}{100}, \frac{(\epsilon m(n))^{10}}{2^{2n}} \right\}, \end{equation}
where $m(n) = \min\{m^{i}_{n} : i \in I_{n}\}$. 

After the definition of $\nu_{n + 1}^{i}$ has been made for every segment $J^{n}_{i}$, we define the measure $\mu_{n + 1}$ as 
\begin{displaymath} \mu_{n + 1} := \sum_{i \in I_{n}} \sum_{t \in D_{n}} [(a_{i}^{n} + t,c_{i}^{n}) + \nu^{i}_{n + 1}]. \end{displaymath}
The notation $(a_{i}^{n} + t,c_{i}^{n}) + v^{i}_{n + 1}$ simply means the measure $\nu^{i}_{n + 1}$ translated by the vector $(a_{i}^{n} + t,c_{i}^{n})$. With this definition, $\mu_{n + 1}$ is a probability measure:
\begin{displaymath} \mu_{n + 1}(\R^{2}) = \sum_{i \in I_{n}} \sum_{t \in D_{n}} \frac{\omega_{i}^{n}d_{n}}{|D_{n}|} = \sum_{i \in I_{n}} \omega_{i}^{n}d_{n} = \mu_{n}(\R^{2}) = 1. \end{displaymath}
For a given interval $J^n_i$, $i \in I_n$, the choice of $N_{n+1}^i$ was made so that for every $t \in D_n$ we have
\begin{align*}
[(a_{i}^{n} + t,c_{i}^{n}) + \nu^{i}_{n + 1}](J^n_i) = m^n_i - N^{i}_{n+1} \cdot \frac{\epsilon}{2^n} \cdot r_{n+1,0} = \frac{m^n_i}{2}.
\end{align*}
In particular, we have
\begin{equation}\label{mestran}
\mu_{n+1}(J^n_i) = \sum_{t \in D_{n}} [(a_{i}^{n} + t,c_{i}^{n}) + \nu^{i}_{n + 1}](J^n_i) = |D_n| \frac{m^n_i}{2} = \frac{\mu_n(J^n_i)}{2}.
\end{equation}

If the numbers $r_{n,0}$ converge to zero quickly enough as $n \to \infty$, one can check that the measures $\mu_{n}$ converge weakly to a limit probability measure on $[0,1]^{2}$, which is our final measure $\mu$. 

Before proving (i) and (ii), we make a few useful remarks about the structure of $\mu$ and its support. First, the support of each measure $\mu_{n + 1}$ is contained in a finite union of closures of dyadic squares of side-length $r_{n + 1,0}$, namely 
\begin{displaymath} \tilde{\calQ}_{n + 1} = \{(a_{i}^{n} + t,c_{i}^{n}) + [0,r_{n + 1,0}]^{2} : i \in I_{n}, t \in D_{n}\}, \end{displaymath}
see Figure \ref{fig2}. The same is true for the measure $\mu$, and in fact we even have the improved conclusion that the support of $\mu$ is contained in the union of the proper dyadic squares
\begin{displaymath} \calQ_{n + 1} = \{(a_{i}^{n} + t,c_{i}^{n}) + [0,r_{n + 1,0})^{2} : i \in I_{n}, t \in D_{n}\}. \end{displaymath}
This improvement is due to the fact that the squares in $\calQ_{n + 2}$ contained in a square $Q \in \calQ_{n + 1}$ always stay at positive distance from the two "open sides" of the dyadic square $Q$. 

The following observation is an immediate corollary of the previous discussion, and the fact that for a fixed segment $J_{i}^{n}$, the squares $(a_{i}^{n} + t,c_{i}^{n}) + [0,r_{n + 1,0}]^{2}$ lie in a single row "on top" of $J_{i}^{n}$:
\begin{lemma}\label{auxLemma1} If $n \geq 1$ and $J_{i}^{n}$, $i \in I_{n}$, are the segments appearing in the definition of $\mu_{n}$, then $\spt \mu$ is contained in the union of rectangles of the form
\begin{displaymath} R^{n}_{i} = \bigcup_{h \in [0,r_{n + 1,0}]} [J^{n}_{i} + (0,h)], \end{displaymath}
and $\mu(R^{n}_{i}) = \mu_{n}(J^{n}_{i})$. 
\end{lemma}
\begin{figure}[h!]
\begin{center}
\includegraphics[scale = 0.4]{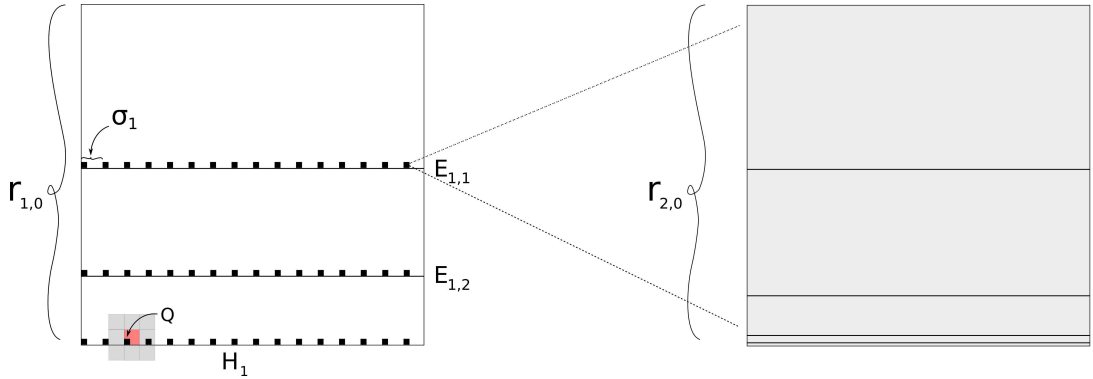}
\caption{The squares in $\calQ_{2}$, and the small square $Q$ appearing in the proof of Lemma \ref{auxLemma2}. The grey area around $Q$ is $3Q$.}\label{fig2}
\end{center}
\end{figure}
Since the side-lengths $r_{n + 1,0}$ are, by \eqref{form9}, far smaller than both the "spacing" $\sigma_{n}$ and the distance $\Delta_{n}$, the following holds for all squares $Q_{1},Q_{2} \in \calQ_{n + 1}$:
\begin{equation}\label{separation} d(Q_{1},Q_{2}) \geq 5r_{n + 1,0}. \end{equation}
This gives the following observation:
\begin{lemma}\label{squareSeq} For every point $x \in \spt \mu$ there exists a unique sequence of dyadic squares $Q_{1},Q_{2},\ldots$ with 
\begin{displaymath} Q_{n} \in \calQ_{n} \quad \text{and} \quad x \in Q_{n}. \end{displaymath}
The square $Q_{n}$ in this sequence will be denoted by $Q_{n}(x)$.
\end{lemma}

Observe that to every square $Q \in \calQ_{n + 1}$, $n \geq 0$, one may canonically associate a sequence 
\begin{displaymath} (r_{n + 1,j})_{j = 0}^{N^{Q}_{n + 1}} \end{displaymath}
and a measure $\nu_{Q}$. For example, for $n = 0$, the collection $\calQ_{n + 1} = \calQ_{1}$ consists of the single square $[0,r_{1,0})^{2}$, so the sequence is $(r_{1,j})_{j = 0}^{N_{1}}$ and the measure is $\nu_{Q} = \mu_{1}$. In general, the measure $\nu_{Q}$ is simply $\nu_{Q} = \mu_{n + 1}|_{Q}$. Then
\begin{displaymath} \nu_{Q} = (a_{i}^{n} + t,c_{i}^{n}) + \nu^{i}_{n + 1} \end{displaymath}
for some $i \in I_{n}$ and $t \in D_{n}$, and the "canonical sequence" $(r_{n + 1,j})$ associated to $Q$ is
\begin{displaymath} (r_{n + 1,j})_{n = 0}^{N^{Q}_{n + 1}} := (r_{n + 1,j})_{j = 0}^{N^{i}_{n + 1}} \end{displaymath}
for this particular $i$. With this notation in mind, we prove another lemma:
\begin{lemma}\label{auxLemma2} Suppose $n \geq 1$, and $Q$ is a dyadic square containing a point $x \in \spt \mu$. Let $Q_{n} = Q_{n}(x) \in \calQ_{n}$ be the square defined in Lemma \ref{squareSeq}. Let $(r_{n,j})$, $0 \leq j \leq N^{Q_{n}}_{n} =: N$ be the sequence define above, and assume that 
\begin{displaymath} \ell(Q) < r_{n,N}/2. \end{displaymath}
Then, all the squares $Q' \in \calQ_{n + 1}$ with $3Q \cap Q' \neq \emptyset$ are horizontal translates of each other, and in fact the measures $\mu_{n + 1}|_{Q'}$ for such $Q'$ are horizontal translates of each other. 
\end{lemma}

\begin{proof} A possible position of the square $Q$ is depicted in the left half of Figure \ref{fig2}. We only prove the lemma in the case $n = 1$; the general case is no different but would require introducing even more notation. In the case $n = 1$, the index set $I_{n} = I_{1}$ is $\{1,\ldots,N_{1}, N_{1}+1\}$, and the segments $J^{1}_{i}$ are the segments $E_{1,j}$ (for $1 \leq j \leq N_{1}$) and $J^{1}_{N_{1} + 1} = H_{1}$. Further, $N = N_{1}$. 

The main assumption $\ell(Q) < r_{1,N_{1}}/2$ means that the side-length of $Q$ is less than half of the minimal (vertical) gap between the segments $E_{1,j}$ and $H_{1}$ -- indeed this minimal gap is the one between $E_{1,N_{1}} = [0,r_{1,0}] \times \{r_{1,N_{1}}\}$ and $H_{1} = [0,r_{1,0}] \times \{0\}$. Since $Q$ is a dyadic square, and the numbers $r_{1,j}$ are all dyadic, this implies that $Q$ is entirely contained in some slab of the form $\R \times [a,b)$, where $a < b$ and $a, b$ are consecutive elements from $\{0,r_{1,N_1},\ldots,r_{1,1},\infty\}$. Assume, for example, that
\begin{displaymath} Q \subset \R \times [0,r_{1,N_{1}}), \end{displaymath}
as in Figure \ref{fig2}. The other cases are handled similarly. Since $Q \cap \spt \mu \neq \emptyset$, and $\spt \mu$ is contained in the union of the squares $\calQ_{2}$, we conclude that either $Q$ is contained in one of the squares in $\calQ_{2}$ (if $\ell(Q) \leq r_{2,0}$) and the lemma is trivial, or else (if $\ell(Q) > r_{2,0}$), $Q$ contain a square $Q' \in \calQ_{2}$ with $Q' \subset \R \times [0,r_{1,N_{1}})$. Such a square $Q'$ must have an edge contained in $\R \times \{0\}$, so also $Q$ has an edge contained in $\R \times \{0\}$. Since $\ell(Q) < r_{1,N_{1}}/2$, it follows that the $y$-coordinate of every point in $3Q$ is strictly smaller than $r_{1,N_{1}}$. Thus, $3Q$ does not meet any of the rectangles $R_{i}^{1}$ for $1 \leq i \leq N_{1}$. On the other hand, the squares $Q' \in \calQ_{2}$ contained in $R^{1}_{N_{1} + 1}$ are, indeed, horizontal translates of each other (inspecting the definition of $\calQ_{2}$, they have the form $Q' = (t,0) + [0,r_{2,1})^{2}$, $t \in D_{1}$). The same is true about the measures: in this particular case they all have the form
\begin{displaymath} (t,0) + \nu((r_{2,j})_{j = 0}^{N^{N_{1} + 1}_{2}},m_{1}^{N_{1} + 1},\epsilon/2), \quad t \in D_{1}. \end{displaymath}
This completes the proof. \end{proof}

\subsection{Lower density vanishes}
We say that a generation $n$, $n \ge 1$, segment $J^n_i$, $i \in I_n$, is of type $E$, if it is not the lower edge of a square in $\calQ_{n}$. The lower edges of squares in $\calQ_{n}$ are called type $H$.
Let $n \ge 1$ and $k \ge 0$. Suppose that $J^n_i$ is of type $E$ and that $x \in R^n_i$ (recall the notation from Lemma \ref{auxLemma1}).

We define $N_n = \max N_n^i$.
If $k = 0$ we clearly have by construction that
$$
\mu_n(B(x, r_{n, N_n}/2)) \le \frac{\epsilon}{2^{n-1}} r_{n, N_n}.
$$
Suppose then that $k = 1$. The ball $B(x, r_{n, N_n}/2)$ can intersect at most $r_{n, N_n} / \sigma_n$ squares in $\calQ_{n + 1}$. Therefore, we have
$$
\mu_{n+1}(B(x, r_{n, N_n}/2)) \le \frac{r_{n, N_n}}{\sigma_n} \frac{\frac{\epsilon}{2^{n-1}} r_{n, 0}}{|D_n|} = \frac{\epsilon}{2^{n-1}} r_{n, N_n}.
$$
The same proof yields with all $k \ge 0$ that
$$
\mu_{n+k}(B(x, r_{n, N_n}/2)) \le  \frac{\epsilon}{2^{n-1}} r_{n, N_n}.
$$
Letting $k \to \infty$ we get (recall that $\mu(G) \le \liminf_{m \to \infty} \mu_m(G)$ for open sets $G$):
$$
\mu(B(x, r_{n, N_n}/2)) \le  \frac{\epsilon}{2^{n-1}} r_{n, N_n}.
$$
We can conclude that
$$
\Theta_*^1(\mu, x) = 0,
$$
if $x$ belongs to infinitely many $R^n_i$, where the corresponding intervals $J^n_i$ are of type $E$.

Next, let us prove that $\mu(J^n_i) = 0$ for all $n \ge 1$ and $i \in I_n$.
Indeed, notice that for every $k \ge 1$ we have
\begin{align*}
\mu(J^n_i) &\le \mu\Bigg( \bigcup_{h \in [0, r_{n+k, N_{n+k}})}  [J^n_i + (0,h) ] \Bigg) \\
& = \mu_{n+k}\Bigg( \bigcup_{h \in [0, r_{n+k, N_{n+k}}) }  [J^n_i + (0,h) ] \Bigg) = \mu_{n+k}(J^n_i) \le \frac{\mu_n(J^n_i)}{2^k} \le 2^{-k}.
\end{align*}
Here we used how the measures are constructed and \eqref{mestran}.

Define
$$
A = \bigcup_{n_0 \ge 1} \bigcap_{n \geq n_{0}} \{x \in \spt\mu\colon\, x \in R^n_i, \textup{ where } J^n_i \textup{ is of type } H\}.
$$
To conclude that $\Theta_*^1(\mu, x) = 0$ for $\mu$-a.e. $x$, it is enough to prove that $\mu(A) = 0$. This follows from the inclusion
\begin{displaymath} A \subset \bigcup_{n \geq 1} \bigcup_{i \in I_{n}} J_{i}^{n}, \end{displaymath}
and the fact that $\mu(J_{i}^{n}) = 0$ for all $n,i$. 

\subsection{The estimate for Jones' square function} In this subsection, we prove that
\begin{equation}\label{form27} \sum_{Q \in \calD} \beta^{2}_{2,\mu}(3Q) \frac{\ell(Q)}{\mu(Q)} \, 1_{Q}(x) \lesssim \epsilon, \quad x \in \spt \mu, \end{equation}
if the parameters in the construction of $\mu$ are chosen appropriately. Fix $x \in \spt \mu$, and let $Q_{n} := Q_{n}(x)$ be the squares defined in Lemma \ref{squareSeq}. As explained above Lemma \ref{auxLemma2}, the square sequence $(Q_{n})_{n = 1}^{\infty}$ gives rise to a sequence of numbers $N_{n} := N_{n}^{Q_{n}}$, $n \geq 1$, and associated sequence of sequences
\begin{displaymath} (r_{n,j})_{j = 0}^{N_{n}} = (r_{n,j})_{j = 0}^{N^{Q_{n}}_{n}}, \quad n \geq 1. \end{displaymath}
Heuristically, these sequences tell us, what kind of "pictures" we see as we zoom to $x$ along the rapidly shrinking dyadic squares $Q_{n}$. The challenge of the proof below will be to handle the $\beta$-numbers for dyadic squares $Q \ni x$, whose side-length lies (far) between the side-lengths of consecutive squares $Q_{n}$ and $Q_{n + 1}$. More precisely, we aim to prove that
\begin{equation}\label{form110} \sum_{r_{n,N_{n}}/2 \leq \ell(Q) < r_{n - 1,N_{n - 1}}/2} \beta^{2}_{2,\mu}(3Q) \frac{\ell(Q)}{\mu(Q)} \, 1_{Q}(x) \lesssim \epsilon/2^{n} \end{equation}
for $n \geq 1$, where we agree that $r_{0,N_{0}} = 1$. This still leaves the sum over the squares with $\ell(Q) \geq 1/2$, but this is easy and will be treated at the end of the proof. The upper bound $\ell(Q) < r_{n - 1,N_{n - 1}}/2 \leq r_{n - 1,0}$ guarantees (recalling the separation condition \eqref{separation}) that every square appearing in the summation only meets one of the squares in $\calQ_{n - 1}$, namely $Q_{n - 1}$. Moreover, since $\ell(Q) < r_{n - 1,N_{n - 1}}/2$, Lemma \ref{auxLemma2} guarantees that the squares in
\begin{displaymath} \calQ_{n}(3Q) := \{Q' \in \calQ_{n} : Q' \cap 3Q \neq \emptyset\}, \end{displaymath}
are horizontal translates of each other, and the same is true for the measures $\mu_{n}|_{Q'}$ with $Q' \in \calQ_{n}(3Q)$. In particular, these measures have a common mass $m$, and by construction, $\mu(Q') = \mu_{n}(Q') = m$ for all $Q' \in \calQ_{n}(3Q)$.

We first prove \eqref{form110} for the dyadic squares $Q$ with $\ell(Q)$ in the restricted range
\begin{equation}\label{form16} r_{n,0} \leq \ell(Q) < r_{n - 1,N_{n - 1}}/2. \end{equation}
\begin{figure}[h!]
\begin{center}
\includegraphics[scale = 0.8]{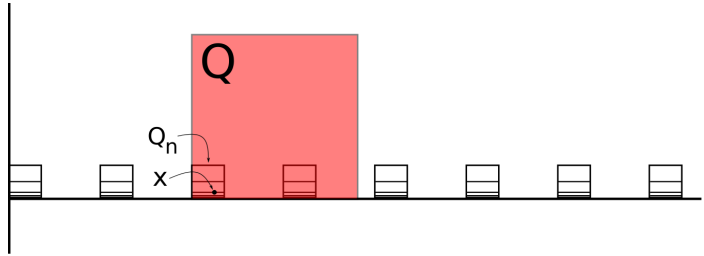}
\caption{The case $r_{n,0} \leq \ell(Q) < r_{n - 1,N_{n - 1}}/2$.}\label{fig4}
\end{center}
\end{figure}
This case is depicted in Figure \ref{fig4}.
Since $x \in Q$ and $\ell(Q) \geq r_{n,0}$, the square $Q$ contains the dyadic square of side-length $r_{n,0}$ containing $x$, namely $Q_{n}$. This gives $\mu(Q) \geq \mu(Q_{n}) = m$. If $q := \card \calQ_{n}(3Q)$, then in fact
\begin{equation}\label{form17} \mu(3Q) \geq \mu(Q) \gtrsim qm, \end{equation}
which follows easily from the inequality $\mu(Q) \geq m$, and the fact that the squares in $\calQ_{n}(3Q)$ are horizontal translates of each other. In order to make further progress in estimating the number $\beta_{2,\mu}^{2}(3Q)$, we need to understand -- and set up notation on -- how $\mu$ looks like inside the squares $Q' \in \calQ_{n}(3Q)$. By definition, $\mu_{n}$ restricted to any square $Q' \in \calQ_{n}(3Q)$ has the form
\begin{displaymath} \mu_{n}|_{Q'} = x_{Q'} + \nu((r_{n,j})^{N_{n}}_{j = 0},m,\epsilon/2^{n - 1}) = x_{Q'} + \sum_{j = 1}^{N_{n}} \frac{\epsilon}{2^{n - 1}} \cdot \calH^{1}|_{E_{n,j}} + \Theta_{n} \cdot \calH^{1}|_{H_{n}}, \end{displaymath} 
where $x_{Q'}$ is the lower left corner point of $Q'$, the $E_{n,j}$'s and $H_{n}$ are the line segments of length $r_{n,0}$ defined in Definition \ref{auxMeasure}, and $\Theta_{n} = \Theta_{n}(m)$ solves the equation 
\begin{displaymath} N_{n}(\epsilon/2^{n - 1})r_{n,0} + \Theta_{n}r_{n,0} = m. \end{displaymath} 
As pointed out in Lemma \ref{auxLemma1}, the support of $\mu|_{Q'}$ lies close to the support of $\mu_{n}$: if 
\begin{displaymath} \tilde{E}_{n,j} := \bigcup_{h \in [0,r_{n + 1,0}]} [E_{n,j} + (0,h)] \end{displaymath}
and similarly 
\begin{displaymath}  \tilde{H}_{n} := \bigcup_{h \in [0,r_{n + 1,0}]} [H_{n} + (0,h)], \end{displaymath}
then
\begin{displaymath} \spt (\mu|_{Q'}) \subset x_{Q'} + \bigcup_{j = 1}^{N_{n}} \tilde{E}_{n,j} \cup \tilde{H}_{n} \end{displaymath}
for $Q' \in \calQ_{n}(3Q)$. Recall that the measures $\mu_{Q'}$ with $Q' \in \calQ_{n}(3Q)$ are all horizontal translates of each other: this implies that, for each fixed $1 \leq j \leq N_{n}$, the segments $x_{Q'} + E_{n,j}$ are contained on a single line $L_{j}$. Similarly, the segments $x_{Q'} + H_{n}$ are contained on a single line $L = L_{N_{n} + 1}$, and this line $L$ will be used to estimate the number $\beta_{2,\mu}^{2}(3Q)$. We use the following bounds:
\begin{displaymath} d(y,L) \sim r_{n,j}, \qquad y \in x_{Q'} + \tilde{E}_{n,j}, \end{displaymath}
and 
\begin{displaymath} d(y,L) \leq r_{n + 1,0}, \qquad y \in x_{Q'} + \tilde{H}_{n}, \end{displaymath}
for any $Q' \in \calQ_{n}(3Q)$. Consequently, recalling \eqref{form17},
\begin{align*} 
&\beta_{2,\mu}^{2}(3Q) \\
 & \lesssim \frac{1}{qm} \sum_{Q' \in \calQ_{n}(3Q)} \int_{Q'} \left(\frac{d(y,L)}{\ell(Q)}\right) \, d\mu(y)\\
& \lesssim \frac{1}{qm} \sum_{Q' \in \calQ_{n}(3Q)} \left[ \int_{x_{Q'} + \tilde{H}_{n}} \left(\frac{r_{n + 1,0}}{\ell(Q)} \right)^{2} \, d\mu(y) + \sum_{j = 1}^{N_{n}} \int_{x_{Q'} + \tilde{E}_{n,j}} \left(\frac{r_{n,j}}{\ell(Q)}\right)^{2} \, d\mu(y) \right]. 
\end{align*}
We estimate the integrals separately. First, the trivial bound $\mu(x_{Q'} + \tilde{H}^{n}) \leq m$ gives
\begin{displaymath} \int_{x_{Q'} + \tilde{H}_{n}} \left(\frac{r_{n + 1,0}}{\ell(Q)}\right)^{2} \, d\mu(y) \leq m \left(\frac{r_{n + 1,0}}{\ell(Q)} \right)^{2}. \end{displaymath}
For $1 \leq j \leq N_{n}$ we have $\mu(x_{Q'} + \tilde{E}_{n,j}) = \mu_{n}(x_{Q'} + E_{n,j}) \leq \ell(Q') = r_{n,0}$, so that
\begin{displaymath} \sum_{j = 1}^{N_{n}} \int_{x_{Q'} + \tilde{E}_{j}^{n}} \left(\frac{r_{n,j}}{\ell(Q)}\right)^{2} \, d\mu(y) \leq \frac{r_{n,0}}{\ell(Q)^{2}} \sum_{j = 1}^{N_{n}} r_{n,j}^{2} \lesssim \frac{r_{n,0}^{3}}{\ell(Q)^{2}}. \end{displaymath} 
Altogether,
\begin{displaymath} \beta_{2,\mu}^{2}(3Q) \lesssim \left(\frac{r_{n + 1,0}}{\ell(Q)}\right)^{2} + \frac{1}{m}\frac{r_{n,0}^{3}}{\ell(Q)^{2}}, \end{displaymath}
so
\begin{align*} \sum_{r_{n,0} \leq \ell(Q) < r_{n - 1, N_{n - 1}}/2} & \beta^{2}_{2,\mu}(3Q)\frac{\ell(Q)}{\mu(Q)}1_{Q}(x)\\
& \lesssim \sum_{r_{n,0} \leq \ell(Q) < r_{n - 1,N_{n - 1}}/2} \left[ \frac{r_{n + 1,0}^{2}}{m \cdot \ell(Q)} + \frac{r_{n,0}^{3}}{m^{2} \cdot \ell(Q)} \right]1_Q(x)\\
& \lesssim \frac{r_{n + 1,0}^{2}}{m \cdot r_{n,0}} + \frac{r_{n,0}^{2}}{m^{2}} \leq \frac{r_{n + 1,0}^{2}}{m(n - 1) \cdot r_{n,0}} + \frac{r_{n,0}^{2}}{m(n - 1)^{2}}. \end{align*} 
Recalling the size condition \eqref{form9} for $r_{n,0}$ (the definition of $m(n - 1)$ can also be found under \eqref{form9}), the right hand side is certainly $\lesssim \epsilon/2^{n}$, as required by \eqref{form110}.

The next goal is to prove that
\begin{equation}\label{form11} \sum_{r_{n,N_{n}}/2 \leq \ell(Q) < r_{n,0}} \beta^{2}_{2,\mu}(3Q) \frac{\ell(Q)}{\mu(Q)} \, 1_{Q}(x) \lesssim \frac{\epsilon}{2^{n}}. \end{equation}
In this case, since $\ell(Q) < r_{n,0}$ and the relevant squares $Q$ contain the point $x \in Q_{n} \in \calQ_{n}$, we have in fact $Q \subset Q_{n}$, and also $3Q$ meets no other squares in $\calQ_{n}$ besides $Q_{n}$. To simplify notation, we assume that
\begin{displaymath} Q_{n} = [0,r_{n,0})^{2}. \end{displaymath}
This has the benefit that the generation $n$ segments $E_{j} = E_{n,j}$ and $H = H_{n}$ inside $Q_{n}$ have the following simple expressions:
\begin{displaymath} E_{j} = [0,r_{n,0}] \times \{r_{n,j}\} \quad \text{and} \quad H = [0,r_{n,0}] \times \{0\}. \end{displaymath}
With this notation, recall that
\begin{displaymath} \mu_{n}|_{Q_{n}} = \nu((r_{n,j})_{j = 0}^{N_{n}},m,\epsilon/2^{n - 1}) = \sum_{j = 1}^{N_{n}} \frac{\epsilon}{2^{n - 1}} \cdot \calH^{1}|_{E_{j}} + \Theta \cdot \calH^{1}|_{H}, \end{displaymath}
where $\Theta = \Theta_{n}(m)$ solves the equation
\begin{displaymath} N_{n}(\epsilon/2^{n - 1})r_{n,0} + \Theta r_{n,0} = m. \end{displaymath} 
Recall from \eqref{form20} and \eqref{form18} that the number $N_{n}$ satisfies
\begin{displaymath} N_{n} = \frac{2^{n - 2}m}{\epsilon r_{n,0}}, \end{displaymath}
so that 
\begin{displaymath} \Theta = \frac{m}{2r_{n,0}} \sim \frac{m}{r_{n,0}}. \end{displaymath} 
We also re-introduce the notation for the rectangles
\begin{displaymath} \tilde{E}_{j} := \bigcup_{h \in [0,r_{n + 1,0}]} [E_{j} + (0,h)] \quad \text{and} \quad \tilde{H} := \bigcup_{h \in [0,r_{n + 1,0}]} [H + (0,h)], \end{displaymath}
so that the support of $\mu|_{Q_{n}}$ is contained in the union of these rectangles, and $\mu(\tilde{E}_{j}) = \mu_{n}(E_{j}) = (\epsilon/2^{n - 1})r_{n,0}$ and $\mu(\tilde{H}) = \mu_{n}(H) = \Theta r_{n,0}$. In particular, the point $x \in \spt \mu \cap Q_{n}$ is contained in one of these rectangles, and the proof splits accordingly.
\subsubsection{The case where $x \in \tilde{H}$} In this case, the square $Q_{n + 1} \ni x$ has the form $[v,v + r_{n + 1,0}) \times [0,r_{n + 1,0})$ for some $0 \leq v < r_{n,0} - \sigma_{n}$, so $x = (t,h)$ with $t \in [0,r_{n,0})$ and 
\begin{displaymath} h < r_{n + 1,0} < r_{n,N_{n}}/2, \end{displaymath}
assuming \eqref{form9}. Now, if $Q$ is any dyadic square with $x \in Q$ and $r_{n,N_{n}}/2 \leq \ell(Q) < r_{n,0}$, this implies that the lower edge of $Q$ is contained on $H$. Hence,
\begin{equation}\label{form13} \mu_{n}(3Q) \geq \mu_{n}(Q) \geq \mu_{n}(Q \cap H) = \Theta \cdot \calH^{1}(Q \cap H) = \Theta \cdot \ell(Q) \sim \frac{m \ell(Q)}{r_{n,0}}. \end{equation}
The very same estimate holds for $\mu$, because the "spacing" $\sigma_{n}$ of the squares in $\calQ_{n + 1}$, which are contained in $\tilde{H}$, satisfies 
\begin{displaymath} \sigma_{n} = \frac{\Delta_{n}}{1000} \leq \frac{r_{n,N_{n}}}{1000} \leq \frac{\ell(Q)}{500}, \end{displaymath}
so many such squares are contained in $Q$. 

 We start proving the estimate \eqref{form11} by summing over the dyadic squares $Q$ with $x \in Q$ and $r_{n,N_{n}}/2 \leq \ell(Q) \leq c r_{n,N_{n} - 1}$, where $c > 0$ is so small that $3Q \cap \tilde{E}_{i} = \emptyset$ for all $1 \leq i \leq N_{1} - 1$. We bound the number $\beta_{2,\mu}^{2}(3Q)$ from above by testing with the line $L = \R \times \{0\}$, and we use the following estimates:
\begin{equation}\label{form23} d(y,L) \sim r_{n,j}, \qquad y \in \tilde{E}_{j} \end{equation}
and 
\begin{equation}\label{form24} d(y,L) \leq r_{n + 1,0}, \qquad y \in \tilde{H}. \end{equation}
Finally, we also need to know that $\mu(3Q \cap \tilde{E}_{i}) \lesssim (\epsilon/2^{n}) \ell(Q)$ and $\mu(3Q \cap \tilde{H}) \lesssim \Theta \ell(Q)$, which follow from the same (trivial) estimates for $\mu_{n}$, and the fact that the "spacing" $\sigma_{n}$ is small enough compared to $\ell(Q)$. Putting all this information together gives
\begin{align} \beta_{2,\mu}^{2}(3Q) & \lesssim \frac{r_{n,0}}{m\ell(Q)} \int_{3Q \cap \spt \mu} \left(\frac{d(y,L)}{\ell(Q)}\right) \, d\mu(y) \label{form21}\\
& \lesssim \frac{r_{n,0}}{m\ell(Q)} \left[ \int_{3Q \cap \tilde{H}} \left(\frac{r_{n + 1,0}}{\ell(Q)} \right)^{2} \, d\mu(y) + \sum_{i = 1}^{N_{n}} \int_{3Q \cap \tilde{E}_{i}} \left(\frac{r_{n,i}}{\ell(Q)}\right)^{2} \, d\mu(y) \right]. \notag \end{align}
Since $3Q \cap \tilde{E}_{j} = \emptyset$ for $1 \leq j \leq N_{n} - 1$, the estimate further becomes
\begin{displaymath} \beta_{2,\mu}^{2}(3Q) \lesssim \frac{r_{n,0}}{m\ell(Q)} \left[ \frac{\Theta r_{n + 1,0}^{2}}{\ell(Q)} + \frac{\epsilon}{2^{n}} \frac{r_{n,N_{n}}^{2}}{\ell(Q)} \right], \end{displaymath}
and consequently, recalling that also $\mu(Q) \gtrsim (m/r_{n,0})\ell(Q)$, 
\begin{align*} \sum_{r_{n,N_{n}}/2 \leq \ell(Q) \leq cr_{n,N_{n} - 1}} & \beta_{2,\mu}^{2}(3Q) \frac{\ell(Q)}{\mu(Q)} 1_{Q}(x)\\
& \lesssim \sum_{r_{n,N_{n}}/2 \leq \ell(Q) \leq cr_{n,N_{n} - 1}} \frac{r_{n,0}^{2}}{m^{2} \ell(Q)} \left[ \frac{\Theta r_{n + 1,0}^{2}}{\ell(Q)} + \frac{\epsilon}{2^{n}} \frac{r_{n,N_{n}}^{2}}{\ell(Q)} \right]1_Q(x)\\
& \lesssim \frac{r^{2}_{n,0}}{m^{2}} \cdot \left[\frac{\Theta r_{n + 1,0}^{2}}{r_{n,N_{n}}^{2}} + (\epsilon/2^{n}) \right]. \end{align*}
This expression looks rather complicated, but the rapid decay \eqref{form9} of the sequence $(r_{n,0})_{n \in \N}$ guarantees that the expression in brackets can be bounded by $1$, so that the whole sum is bounded by $\lesssim r_{n,0}^{2}/m^{2}$. 

Next, we perform a similar estimate for those dyadic squares $Q \ni x$ such that $c r_{n,j} \leq \ell(Q) \leq c r_{n,j - 1}$ for some $1 \leq j \leq N_{n} - 1$. This is fairly similar to the previous bound: the upper bound $\ell(Q) \leq c r_{n,j - 1}$ guarantees that $3Q \cap \tilde{E}_{i} = \emptyset$ for $1 \leq i \leq j - 1$, and hence the sum in the $\beta$-number estimate \eqref{form21} only counts the numbers $r_{n,i}$ with $i \geq j$. This yields
\begin{displaymath} \beta_{2,\mu}^{2}(3Q) \lesssim \frac{r_{n,0}}{m\ell(Q)} \left[ \frac{\Theta r_{n + 1,0}^{2}}{\ell(Q)} + \frac{\epsilon}{2^{n}} \frac{r_{n,j}^{2}}{\ell(Q)} \right], \end{displaymath}
and consequently 
\begin{displaymath} \sum_{cr_{n,j} \leq \ell(Q) \leq cr_{n,j - 1}} \beta_{2,\mu}^{2}(3Q) \frac{\ell(Q)}{\mu(Q)} 1_{Q}(x) \lesssim \frac{r^{2}_{n,0}}{m^{2}} \cdot \left[\frac{\Theta r_{n + 1,0}^{2}}{r_{n,j}^{2}} + (\epsilon/2^{n}) \right]. \end{displaymath}
As before, we simply estimate this by $\lesssim r_{n,0}^{2}/m^{2}$. Putting the various intervals of $\ell(Q)$ together gives
\begin{equation}\label{form25} \sum_{r_{n,N_{n}}/2 \leq \ell(Q) \leq c r_{n,0}} \beta_{2,\mu}^{2}(3Q) \frac{\ell(Q)}{\mu(Q)} 1_{Q}(x) \lesssim \frac{N_{n} r_{n,0}^{2}}{m^{2}} \leq \frac{2^{n - 2}r_{n,0}}{\epsilon m(n - 1)} \leq \frac{\epsilon}{2^{n}} \end{equation}
by the decay asumption \eqref{form9}. This nearly completes the case $x \in \tilde{H}$, except for the dyadic squares $Q \ni x$ with $c r_{n,0} \leq \ell(Q) < r_{n,0}$. They satisfy $\mu(3Q) \geq \mu(Q) \sim m$, and the $\beta$-number $\beta_{2,\mu}^{2}(3Q)$ satisfies
\begin{displaymath} \beta_{2,\mu}^{2}(3Q) \lesssim \left(\frac{r_{n + 1,0}}{r_{n,0}} \right)^{2} + \frac{\epsilon}{2^{n}} \leq 1, \end{displaymath}
reviewing \eqref{form21} in this case. Consequently,
\begin{equation}\label{form26} \sum_{cr_{n,0} \leq \ell(Q) < r_{n,0}} \beta_{2,\mu}^{2}(3Q) \frac{\ell(Q)}{\mu(Q)} \, 1_{Q}(x) \lesssim \frac{r_{n,0}}{m} \leq \frac{\epsilon}{2^{n}}, \end{equation}
as desired. This completes the proof of the case $x \in \tilde{H}$. 

\subsubsection{The case where $x \in \tilde{E}_{j}$ for some $1 \leq j \leq N_{n}$} Let $Q$ be the dyadic square of side-length $r_{n,N_{n}}/2 \leq \ell(Q) < r_{n,0}$ containing $x$. Now the proof of \eqref{form11} splits into three sub-cases: (a) $r_{n,N_{n}}/2 \leq \ell(Q) < r_{n,j}$, (b) $\ell(Q) = r_{n,j}$, and (c) $r_{n,j} < \ell(Q) < r_{n,0}$. The estimates needed for (c) are very familiar from the previous case, and we will only sketch them briefly at the end. The cases (a) and (b) present new phenomena, and we start with (a).  

In (a), since $\ell(Q) \leq r_{n,j}$ and $r_{n,j}$ is a dyadic number, there are two possibilities: the $y$-coordinate of every point in $Q$ is $\geq r_{n,j}$, or the $y$-coordinate of every point in $Q$ is $< r_{n,j}$. But $x \in Q$ has $y$-coordinate $\geq r_{n,j}$ (since $\tilde{E}_{j}$ is entirely contained "on top" of the line $E_{j}$), so the former possibility holds. This, and $\ell(Q) \leq r_{n,j}/2$, imply that 
\begin{equation}\label{form22} 3Q \cap \spt \mu \subset \tilde{E}_{j}. \end{equation}
It follows, heuristically, that the support of $\mu$ looks extremely flat in the square $3Q$ with $\ell(3Q) \geq r_{n,N_{n}} \gg r_{n + 1,0} = \text{ width}(\tilde{E}_{j})$. 

To be precise, observe first that the lower edge of $Q$ is contained on $E_{j}$, which follows from $x \in Q$ (which forces $Q$ to lie at distance $\sim r_{n + 1,0}$ to $E_{j}$) and $\ell(Q) \geq r_{n,N_{n}}/2$ (which implies that the side-length of $Q$ is far larger than $d(Q,E_{j})$. Since the "spacing" $\sigma_{n} \leq r_{n,N_{n}}/1000$ of the squares of $\calQ_{n + 2}$ inside $\tilde{E}_{j}$ is far smaller than $\ell(Q)$, this implies that
\begin{displaymath} \mu(3Q) \geq \mu(Q) \sim \mu_{n}(Q) = \mu_{n}(Q \cap E_{j}) \sim \frac{\epsilon \ell(Q)}{2^{n}}. \end{displaymath}
The converse inequality also holds by \eqref{form22}. Hence, if $L$ is the line containing $E_{j}$, we have
\begin{displaymath} \beta_{2,\mu}^{2}(3Q) \lesssim \frac{2^{n}}{\epsilon \ell(Q)} \int_{\tilde{E}_{j} \cap 3Q} \left(\frac{d(y,L)}{\ell(3Q)} \right)^{2} \, d\mu(y) \lesssim \left(\frac{r_{n + 1,0}}{\ell(Q)}\right)^{2}, \end{displaymath}
and consequently
\begin{displaymath} \sum_{r_{n,N_{n}}/2 \leq \ell(Q) < r_{n,j}} \beta_{2,\mu}^{2}(3Q) \frac{\ell(Q)}{\mu(Q)} 1_{Q}(x) \lesssim \frac{2^{n} r_{n + 1,0}^{2}}{\epsilon} \sum_{r_{n,N_{n}}/2 \leq \ell(Q) < r_{n,j}}\frac{1_Q(x)}{\ell(Q)^{2}} \lesssim \frac{2^{n} r_{n + 1,0}^{2}}{\epsilon r_{n,N_{n}}^{2}}. \end{displaymath}
Assuming the rapid decay \eqref{form9}, the right hand side is smaller than $\epsilon/2^{n}$, and this case is ready. 

The next case is (b), namely that $x \in Q$ and $\ell(Q) = r_{n,j}$. Exactly as in the previous case, these conditions imply that the $y$-coordinate of every point in $Q$ is $\geq r_{n,j}$, and the lower edge of $Q$ is contained on $E_{j}$. The difference to the previous case is now that $3Q$ sees $H$: since $r_{n,j} = \ell(Q)$ is a dyadic number, and the lower edge of $Q$ lies on $E_{j}$ we have $\calH^{1}(3Q \cap H) \sim \ell(Q)$. This implies that
\begin{displaymath} \mu(3Q) \sim \mu_{n}(3Q) \sim \mu_{n}(3Q \cap H) \sim \Theta \ell(Q) \sim \frac{m\ell(Q)}{r_{n,0}} = \frac{m r_{n,j}}{r_{n,0}}. \end{displaymath}
The square $Q$ itself does not meet $H$, so we will need to make do with a weaker lower bound for $\mu(Q)$:
\begin{displaymath} \mu(Q) \sim \mu_{n}(Q) \geq \mu_{n}(Q \cap E_{j}) = \frac{\epsilon r_{n,j}}{2^{n}}. \end{displaymath}
To estimate the number $\beta^{2}_{2,\mu}(3Q)$, we use the line $L = \R \times \{0\}$. Observing that $3Q \cap \tilde{E}_{i} = \emptyset$ for $1 \leq i \leq j - 1$, and that $\mu(3Q \cap \tilde{E}_{i}) \lesssim (\epsilon/2^{n}) r_{n,j}$ for $j \leq i \leq N_{n}$, and recalling the distance estimates \eqref{form23}-\eqref{form24}, we have
\begin{align*} \beta_{2,\mu}^{2}(3Q) & \lesssim \frac{r_{n,0}}{mr_{n,j}}  \left[ \int_{3Q \cap \tilde{H}} \left(\frac{d(y,L)}{r_{n,j}}\right)^{2} \, d\mu(y)+ \sum_{i = j}^{N_{n}} \int_{3Q \cap \tilde{E}_{i}} \left(\frac{d(y,L)}{r_{n,j}} \right)^{2} \, d\mu(y) \right]\\
& \lesssim \frac{r_{n,0}}{mr_{n,j}} \left[\frac{m r_{n,j}}{r_{n,0}} \left(\frac{r_{n + 1,0}}{r_{n,j}}\right)^{2} + \frac{\epsilon}{2^{n}} \sum_{i = j}^{N_{n}} \frac{r_{n,i}^{2}}{r_{n,j}} \right] \sim \left(\frac{r_{n + 1,0}}{r_{n,j}}\right)^{2} + \frac{\epsilon r_{n,0}}{m 2^{n}}. \end{align*}
Thus,
\begin{displaymath} \sum_{\ell(Q) = r_{n,j}} \beta_{2,\mu}^{2}(3Q) \frac{\ell(Q)}{\mu(Q)} 1_{Q}(x) \lesssim \left[\left(\frac{r_{n + 1,0}}{r_{n,j}}\right)^{2} + \frac{\epsilon r_{n,0}}{m(n - 1)2^{n}}\right] \cdot \frac{2^{n}}{\epsilon} \leq \frac{\epsilon}{2^{n}} \end{displaymath}
by the rapid decay assumption \eqref{form9}. This case is complete. 

We arrive at case (c), where $r_{n,j} < \ell(Q) < r_{n,0}$. In particular, $\ell(Q) \geq 2r_{n,j}$. Since $x \in \tilde{E}_{j}$, we have $x = (t,h)$ with $t \in [0,r_{n,0})$
\begin{displaymath} h \leq r_{n,j} + r_{n + 1,0} < 2r_{n,j} \leq \ell(Q). \end{displaymath}
As $x \in Q$, this forces the lower edge of $Q$ to lie on $H$, and we have the good lower bounds \eqref{form13} for both $\mu(Q)$ and $\mu(3Q)$. The rest of case (c) is exactly the same as the proof of the bounds \eqref{form25} and \eqref{form26} in the case $x \in \tilde{H}$, and we do not repeat the details. The proof of the case $x \in \tilde{E}_{j}$ is now complete. 

\subsubsection{Conclusion of the proof of the $\beta$-number estimate} We have now proved the estimate \eqref{form110}, valid for all $x \in \spt \mu$ and $n \geq 1$. This gives
\begin{displaymath} \sum_{\ell(Q) \leq 1/2} \beta_{2,\mu}^{2}(3Q) \frac{\ell(Q)}{\mu(Q)}1_{Q} = \sum_{n = 1}^{\infty} \sum_{r_{n,N_{n}}/2 \leq \ell(Q) < r_{n - 1,N_{n - 1}}/2} \beta^{2}_{2,\mu}(3Q)\frac{\ell(Q)}{\mu(Q)} 1_{Q}(x) \lesssim \epsilon. \end{displaymath} 
So, it remains to prove the same estimate for $Q$ with $x \in Q$ and $\ell(Q) \geq 1/2$, which is straightforward. If $x \in Q$ and $\ell(Q) \geq 1/2$, then $\mu(Q) = 1 = \mu(3Q)$ (at least if $r_{1,0} < 1/2$, which we may assume). We use the line $L = \R \times \{0\}$ to estimate the $\beta$-number. The notation $\tilde{E}_{j}$ and $\tilde{H}$ should be self-explanatory by now:
\begin{displaymath} \beta_{2,\mu}^{2}(3Q) \leq \int_{\tilde{H}} \left(\frac{r_{2,0}}{\ell(Q)^{2}} \right)^{2} \, d\mu + \sum_{j = 1}^{N_{1}} \int_{\tilde{E}} \left(\frac{r_{1,j}}{\ell(Q)} \right)^{2} \, d\mu \lesssim \left(\frac{r_{2,0}}{\ell(Q)} \right)^{2} + \left(\frac{r_{1,0}}{\ell(Q)} \right)^{2}.  \end{displaymath}
Thus,
\begin{displaymath} \sum_{\ell(Q) > 1/2} \beta_{2,\mu}^{2}(3Q) \frac{\ell(Q)}{\mu(Q)} 1_{Q}(x) \leq \sum_{\ell(Q) > 1/2} \left(\frac{r_{2,0}^{2}}{\ell(Q)} + \frac{r_{1,0}^{2}}{\ell(Q)}\right) 1_{Q}(x) \sim r_{2,0}^{2} + r_{1,0}^{2} \leq \epsilon, \end{displaymath} 
if $r_{1,0} \leq \epsilon/2$, which we may assume. This completes the proof of the $\beta$-number estimate \eqref{form27}. 

\section{Proof of the continuous case: Theorem \ref{main3}}
In Theorem \ref{main3}, the measure $\mu$ is precisely the same as in Theorem \ref{main}, so we do not repeat the argument for claim (ii). Before starting the proof of (i), we briefly recall how $\mu$ looks like:
\begin{displaymath} \mu = \sum_{j = 1}^{N} \epsilon' \cdot \calH^{1}|_{E_{j}} + \Theta \cdot \calH^{1}|_{H}, \end{displaymath}
where $\epsilon' = c\epsilon$ for some small absolute constant $c > 0$, $E_{j} = [0,r_{0}] \times \{r_{j}\}$, $H = [0,r_{0}] \times \{0\}$,
\begin{displaymath} r_{0} = (\epsilon')^{3}, \quad N = \frac{1 - \epsilon'}{\epsilon' r_{0}} \quad \text{and} \quad \Theta = \frac{\epsilon'}{r_{0}}. \end{displaymath}
As before, we assume without loss of generality that $(1 - \epsilon')/(\epsilon'r_{0})$ is an integer. For this proof, it is also convenient to define
\begin{displaymath} E_{0} = [0,r_{0}] \times \{r_{0}\}. \end{displaymath}
The sequence $(r_{j})_{j = 1}^{N}$ is assumed to be rapidly decreasing enough, where the rate of decay will depend on $\epsilon$. Some aspects of the construction are depicted in Figure \ref{fig1}.
\begin{figure}[h!]
\begin{center}
\includegraphics[scale = 0.7]{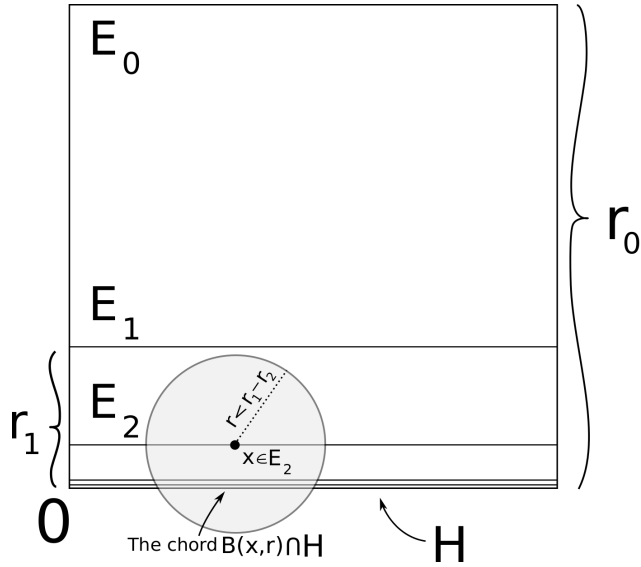}
\caption{The measure $\mu$ and various related objects.}\label{fig1}
\end{center}
\end{figure}
We will now prove (i), namely that
\begin{displaymath} \int_{0}^{\infty} \beta^{2}_{2,\mu}(B(x,r)) \frac{dr}{\mu(B(x,r))} \leq \epsilon \quad \text{for all } x \in \spt \mu. \end{displaymath}
First, suppose $x \in H$ and $r > 0$. If $r < r_N$, we have that $\beta^{2}_{2,\mu}(B(x,r)) = 0$. 
Another easy case is $r \ge r_0$, and we handle it now. Then we have $\mu(B(x,r)) \gtrsim 1$ and, testing with the line $L = \R \times \{0\}$,
\begin{equation}\label{form1}
\beta^{2}_{2,\mu}(B(x,r)) \le \epsilon' \sum_{j=1}^N \int_{B(x,r) \cap E_j} \Big( \frac{r_j}{r}\Big)^2 d\calH^1(y) \lesssim \frac{\epsilon' r_0^3}{r^2},
\end{equation}
so that
$$
\int_{r_0}^{\infty} \beta^{2}_{2,\mu}(B(x,r)) \frac{dr}{\mu(B(x,r))} \lesssim \epsilon' r_0^3 \int_{r_0}^{\infty} r^{-2}\,dr = \epsilon' r_{0}^{2} = (\epsilon')^{7}.
$$
Let now $r \in [r_N, r_0)$, and specifically $r \in [r_{j + 1}, r_j)$ for some $j \in {0, \ldots, N-1}$.
We have $\mu(B(x,r)) \gtrsim \epsilon' r_0^{-1} r$ and, observing that $B(x,r) \cap E_{j} = \emptyset$,
\begin{equation}\label{form3}
\mu(B(x,r)) \beta^{2}_{2,\mu}(B(x,r)) \le \epsilon' \sum_{i=1}^N \int_{B(x,r) \cap E_i} \Big( \frac{r_i}{r}\Big)^2 d\calH^1(y) \lesssim \frac{\epsilon'}{r} \sum_{i:\, r \ge r_i} r_i^2 \lesssim \frac{\epsilon'}{r}r_{j+1}^2,
\end{equation}
so that
$$
\int_{r_{j+1}}^{r_j} \beta^{2}_{2,\mu}(B(x,r)) \frac{dr}{\mu(B(x,r))} \lesssim \int_{r_{j+1}}^{r_j} \frac{r_0^2}{(\epsilon' r)^2} \frac{\epsilon'}{r}r_{j+1}^2 \,dr \lesssim \frac{r_0^2}{\epsilon'}.
$$
This yields
$$
\int_{r_N}^{r_0} \beta^{2}_{2,\mu}(B(x,r)) \frac{dr}{\mu(B(x,r))} \lesssim N \frac{r_0^2}{\epsilon'} \lesssim \frac{r_0}{(\epsilon')^2} = \epsilon'.
$$
Combining the above estimates gives
$$
\int_0^{\infty} \beta^{2}_{2,\mu}(B(x,r)) \frac{dr}{\mu(B(x,r))} \le \epsilon, \qquad x \in H,
$$
by choosing $c$ in $\epsilon' = c\epsilon$ small enough.

Let then $x \in E_{j_{0}}$, $j_{0} = 1, \ldots, N$. We split
\begin{align*}
\int_0^{\infty} & \beta^{2}_{2,\mu}(B(x,r)) \frac{dr}{\mu(B(x,r))} \\
&= \Big(\int_0^{r_{j_{0}} - r_{j_{0}+1}}
+ \int_{r_{j_{0}} - r_{j_{0} + 1}}^{r_{j_{0} - 1} - r_{j_{0}}}
+ \int_{r_{j_{0} - 1} - r_{j_{0}}}^{r_{0} - r_{j_{0}}}
+ \int_{r_{0} - r_{j_{0}}}^{\infty}
\Big)  \beta^{2}_{2,\mu}(B(x,r)) \frac{dr}{\mu(B(x,r))} \\
&= I + II + III + IV,
\end{align*}
and further
$$
III = \sum_{j=1}^{j_{0}-1} \int_{r_{j} - r_{j_{0}}}^{ r_{j - 1} - r_{j_{0}}} \beta^{2}_{2,\mu}(B(x,r)) \frac{dr}{\mu(B(x,r))} = \sum_{j=1}^{j_{0}-1} III_j.
$$

(note that $III = 0$ trivally if $j_{0} = 1$, so we may assume that $j_{0} \geq 2$ while treating $III$). If $0 < r < r_{j_{0}} - r_{j_{0}+1}$, then $B(x,r)$ intersects only one of the line segments of $\spt \mu$, and $\beta_{2,\mu}^{2}(B(x,r)) = 0$. Thus $I = 0$.

Another easy case is, when $r \geq r_{0} - r_{j_{0}}$, which gives the same result as the case $r \geq r_{0}$ above, again using $\mu(B(x,r)) \gtrsim 1$, testing with the line $L = \R \times \{0\}$, and repeating the computation from \eqref{form1}. We get $IV \lesssim \epsilon' r_0^2 = (\epsilon')^{7}$.

The case $II$ is the hardest. So we first deal with $III$ by handling a fixed term $III_j$ for some $1 \le j \le j_0-1$, in particular
\begin{equation}\label{form2} r_{j} - r_{j_{0}} < r < r_{j - 1} - r_{j_{0}}. \end{equation}
We will use the following formula for the length of a \emph{chord}: if $S(z,R_{2})$ is any circle of radius $R_{2} > 0$, and $J$ is a chord at distance $0 < R_{1} \leq R_{2}$ from $z$, then
\begin{equation}\label{form5} \calH^{1}(J) = 2\sqrt{R_{2}^{2} - R_{1}^{2}} \sim \sqrt{R_{2}} \sqrt{R_{2} - R_{1}}. \end{equation} 
Applying this to the chord $B(x,r) \cap H$ of the circle $S(x,r)$, we obtain
\begin{displaymath} \mu(B(x,r)) \geq \mu(B(x,r) \cap H) = \Theta \cdot \calH^{1}(B(x,r) \cap H) \sim \frac{\epsilon'}{r_{0}} \sqrt{r} \sqrt{r - r_{j_{0}}} \sim \frac{\epsilon' r}{r_{0}}. \end{displaymath} 
Here the equivalence $r - r_{j_{0}} \sim r$ follows from the left hand side of \eqref{form2}, which implies that
\begin{displaymath} r - 2r_{j_{0}} \geq r_{j_{0} - 1} - r_{j_{0}} - 2r_{j_{0}} \geq 0, \end{displaymath}
or equivalently $r - r_{j_{0}} \geq r/2$. Thus,
\begin{displaymath} \int_{r_{j} - r_{j_{0}}}^{r_{j - 1} - r_{j_{0}}} \beta_{2,\mu}^{2}(B(x,r)) \frac{dr}{\mu(B(x,r))} \lesssim \int_{r_{j/2}}^{r_{j - 1} - r_{j_{0}}} \frac{r_{0}^{2}}{(\epsilon'r)^{2}} \cdot [\mu(B(x,r))\beta_{2,\mu}^{2}(B(x,r)) ] \, dr, \end{displaymath}
and whenever $0 < r < r_{j - 1} - r_{j_{0}}$, repeating the computation from \eqref{form3} yields
\begin{displaymath} \mu(B(x,r))\beta^{2}_{2,\mu}(B(x,r)) \lesssim \frac{\epsilon'}{r}r_{j}^{2}. \end{displaymath}
Altogether, 
\begin{displaymath} III_j = \int_{r_{j} - r_{j_{0}}}^{r_{j - 1} - r_{j_{0}}} \beta_{2,\mu}^{2}(B(x,r)) \frac{dr}{\mu(B(x,r))} \lesssim \frac{r_{0}^{2}}{\epsilon'}, \end{displaymath}
and summing this estimate for $1 \leq j  \leq j_{0}-1$ gives
\begin{displaymath} III = \sum_{j=1}^{j_{0}-1} III_j \lesssim \frac{r_{0}}{(\epsilon')^{2}} = \epsilon'. \end{displaymath}

It remains to handle the case $II$, where
\begin{equation}\label{form4} r_{j_{0}} - r_{j_{0} + 1} < r \leq r_{j_{0} - 1} - r_{j_{0}}. \end{equation}
We begin with the former task, so fix $r$ as in \eqref{form4}. Observe that
\begin{displaymath} \int_{r_{j_{0}} - r_{j_{0} + 1}}^{r_{j_{0}} - r_{j_{0} + 1} + (\epsilon')^{2}r_{j_{0}}} \beta_{2,\mu}^{2}(B(x,r)) \frac{dr}{\mu(B(x,r))} \leq  \int_{r_{j_{0}} - r_{j_{0} + 1}}^{r_{j_{0}} - r_{j_{0} + 1} + (\epsilon')^{2}r_{j_{0}}} 1 \cdot \frac{dr}{\epsilon' r} \lesssim \frac{(\epsilon')^{2}r_{j_{0}}}{\epsilon' r_{j_{0}}} = \epsilon'. \end{displaymath}
It remains to handle the integration over the range 
\begin{displaymath} r_{j_{0}} - r_{j_{0} + 1} + (\epsilon')^{2}r_{j_{0}} < r < r_{j_{0} - 1} - r_{j_{0}}. \end{displaymath}
Choose $r$ on this interval, and note that, by requiring the decay of the sequence $(r_{i})$ to be so rapid that $r_{j_{0} + 1} < (\epsilon')^{2}r_{j_{0}}/2$, we have
\begin{equation}\label{form6}  r_{j_{0}} + \frac{(\epsilon')^{2}r_{j_{0}}}{2} \leq r < r_{j_{0} - 1} - r_{j_{0}}. \end{equation}
We wish to find a lower bound for $\mu(B(x,r))$, and this is accomplished with the aid of the chord length estimate \eqref{form5}. Namely, repeating an earlier estimate,
\begin{displaymath} \mu(B(x,r)) \geq \mu(B(x,r) \cap H) = \Theta \cdot \calH^{1}(B(x,r) \cap H) \sim \frac{\epsilon'}{r_{0}}\sqrt{r}\sqrt{r - r_{j_{0}}}. \end{displaymath}
Note that the left hand side of \eqref{form6} implies that
\begin{displaymath} r - r_{j_{0}} \gtrsim (\epsilon')^{2}r, \end{displaymath}
because this holds at $r = r_{j_{0}} + (\epsilon')^{2}r_{j_{0}}/2$, the left endpoint of our interval, and the derivative of $r \mapsto  r - r_{j_{0}} - (\epsilon')^{2}r$ is positive. Hence,
\begin{displaymath} \mu(B(x,r)) \gtrsim \frac{(\epsilon')^{2}r}{r_{0}}. \end{displaymath}
Next, we essentially repeat the computation from \eqref{form3}, recalling the upper bound $r < r_{j_{0} - 1} - r_{j_{0}}$, which means that $B(x,r) \cap E_{j_{0} - 1} = \emptyset$:
\begin{displaymath} \mu(B(x,r)) \beta^{2}_{2,\mu}(B(x,r)) \le \epsilon' \sum_{i=1}^N \int_{B(x,r) \cap E_i} \Big( \frac{r_i}{r}\Big)^2 d\calH^1(y) \lesssim \frac{\epsilon'}{r} \sum_{i = j_{0}}^{N} r_i^2 \lesssim \frac{\epsilon'}{r}r_{j_{0}}^2. \end{displaymath}
Combining the estimates above leads to
\begin{displaymath} \int_{r_{j_{0}} - r_{j_{0} + 1} + (\epsilon')^{2}r_{j_{0}}}^{r_{j_{0} - 1} - r_{j_{0}}} \beta_{2,\mu}^{2}(B(x,r)) \frac{dr}{\mu(B(x,r))} \lesssim \frac{r_{0}^{2}}{(\epsilon')^{3}} \int_{r_{j_{0}}}^{r_{j_{0} - 1} - r_{j_{0}}} \frac{r_{j_{0}}^{2}}{r^{3}} \, dr \lesssim \frac{r_{0}^{2}}{(\epsilon')^{3}}. \end{displaymath}
Collecting the estimates we get
$$
II \lesssim \frac{r_{0}^{2}}{(\epsilon')^{3}} + \epsilon' = (\epsilon')^{3} + \epsilon'.
$$

Combining the estimates for the terms $I, \ldots, IV$ we have
$$
\int_0^{\infty} \beta^{2}_{2,\mu}(B(x,r)) \frac{dr}{\mu(B(x,r))} \le \epsilon, \qquad x \in E_{j_{0}},\, j_{0} = 1, \ldots, N,
$$
by choosing $c > 0$ small enough in $\epsilon' = c\epsilon$. The proof of Theorem \ref{main3} is complete. 


\end{document}